\documentclass{article}

\usepackage{epsfig}
 \usepackage{epstopdf}
\usepackage[T1]{fontenc}
\usepackage{geometry}
\usepackage{amsbsy,amsmath,latexsym,amsfonts, epsfig, color, authblk, amssymb, graphics, bm}
\usepackage{epsf,slidesec,epic,eepic}
\usepackage{fancybox}
\usepackage{fancyhdr}
\usepackage{setspace}
\usepackage{nccmath}
\usepackage[colorlinks,linkcolor=black,anchorcolor=black,citecolor=black,hyperindex,CJKbookmarks]{hyperref}
\newenvironment{proof}{{\noindent \bf Proof.}}{\hfill$\Box$\medskip}

\newtheorem{theorem}{Theorem}[section]
\newtheorem{corollary}[theorem]{Corollary}
\newtheorem{lemma}[theorem]{Lemma}

\newtheorem{definition}[theorem]{Definition}

\newtheorem{remark}{Remark}

\def \R{\mathbb{R}}

\def \N{\mathbb{N}}

\def \A{\mathcal{A}}

\begin{document}

\title{ The dimension of irregular set in parameter space\footnotetext {* Corresponding author}
}
\author{  Lixuan Zheng$^\dag$, Min Wu$^\dag$ and Bing Li$^{\dag, *}$\\
\small \it $\dag$ School of Mathematics\\
\small \it South China University of Technology\\
\small \it Guangzhou 510640, P.R. China\\
\small \it E-mails: z.lixuan@mail.scut.edu.cn, wumin@scut.edu.cn and
scbingli@scut.edu.cn}
\date{5th March,2015}
\date{}
\maketitle
\begin{center}
\begin{minipage}{120mm}
{\small {\bf Abstract.}  For any real number $\beta>1$. The $n$th cylinder of $\beta$ in the parameter space $\{\beta\in \R: \beta>1\}$ is a set of real numbers in $(1,\infty)$ having the same first $n$ digits in their $\beta$-expansion of $1$, denote by $I^P_n(\beta)$. We study the quantities which describe the growth of the length of $I^P_n(\beta)$. The Huasdorff dimension of the set of given growth rate of the length of $I^P_n(\beta)$ will be determined in this paper.}
\end{minipage}
\end{center}

\vskip0.5cm

\section{Introduction}

Given a real number $\beta>1.$ Let $T_{\beta}:[0,1] \rightarrow [0,1]$ be the \emph{$\beta$-transformation} which is defined by $$T_{\beta}x = \beta
x-\lfloor\beta x\rfloor,$$ where $\lfloor \beta x\rfloor$  denotes the integer part of $x$. In 1957, R\'{e}nyi \cite{R} has shown that every real number $x$ in $[0,1]$ can be written as a finite or an infinite series by the iteration of $T_{\beta}$ as follows,
\begin{equation}\label{1.1}
x=\frac{\varepsilon_1(x,\beta)}{\beta}+\cdots+\frac{\varepsilon_n(x,\beta)+T_{\beta}^n
x}{\beta^n}=\sum_{n=1}^\infty \frac{\varepsilon_n(x,\beta)}{\beta^n},
\end{equation}
where, for each $n\geq1$, $$\varepsilon_n(x,\beta)=\lfloor \beta T_{\beta}^{n-1}x
\rfloor.$$ And $\varepsilon_n(x,\beta)$ is said to be the \emph{$n$th digit} of $x$. We identify $x$
with its digit sequence $$\varepsilon(x,\beta):=(\varepsilon_1(x,\beta),\ldots,\varepsilon_n(x,\beta),\ldots)$$ and call
the digit sequence $\varepsilon(x,\beta)$ the \emph{$\beta$-expansion} of $x$.  As one of the typical example of monotone one-dimensional dynamical system, the transformation $T_{\beta}$ has drawn much attention, see \cite{B,LW,P,S}, etc.\

Recently, the $n$th cylinders of a real number in the interval $[0,1)$ and some related sets have been studied, see \cite{BW, AB}. Similarly, our work focuses on the study of cylinders of real numbers in the parameter space. Fixed $\beta>1$, we define the $n$th cylinder of $\beta$ in the parameter space as the family of $\beta'\in(1,+\infty)$ whose $\beta'$-expansion of $1$ beginning with $(\varepsilon_1(1,\beta),\ldots,\varepsilon_n(1,\beta))$, i.e.\  $$I^P_n(\beta):=\{\beta'\in(1,+\infty):\varepsilon_1(1,\beta')=\varepsilon_1(1,\beta),\ldots, \varepsilon_n(1,\beta')=\varepsilon_n(1,\beta)\}.$$ Schmeling \cite{S} showed that such cylinder in the parameter space is an interval. The length of the $n$th cylinder in the parameter space of $\beta$ is denoted as $|I^P_n(\beta)|$. Without any confusion in this paper, we denote the left endpoint of $I^P_n(\beta)$ as $\underline{\beta}_n$ and the right endpoint of $I^P_n(\beta)$ as $\overline{\beta}_n$ accordingly for every $\beta>1$.

As a matter of fact, the definition above coincides with the definition of the $n$th cylinder in the parameter space which first introduced by J. Schmeling \cite{S} for an admissible block and he gave the upper bound of the length $|I^P_n(\beta)|$, see Theorem \ref{S} for more details. After that, B. Li et al. studied the lower bound of $|I^P_n(\beta)|$ in \cite{LP}, and gave an evaluation of it, see Theorem \ref{EST}. Applying the above results, a simple fact of the left and right endpoint of  $I^P_n(\beta)$ is soon established as follows.
\begin{theorem}\label{LIMIT}
Let $\beta \in (1,+\infty)$. Then, $\underline{\beta}_n$ is increasing to the limit $\beta$ and  $\overline{\beta}_n$ is decreasing to the limit $\beta$ as n go to infinity.
\end{theorem}

The above theorem indicates that the growth of the length $|I^P_n(\beta)|$ as $n\rightarrow +\infty$ relies strongly on $\beta$. We introduce and study the following quantities which describe the rate of the growth of $|I^P_n(\beta)|$. For any $\beta \in (1,+\infty)$, we define the upper and lower density of $\beta$ as follows,
\begin{equation}\label{inf}
\underline{D}(\beta) = \liminf_{n\rightarrow \infty}\frac{-\log_\beta
|I^P_n(\beta)|}{n},
\end{equation}
\begin{equation}\label{sup}
\overline{D}(\beta) = \limsup_{n\rightarrow \infty}\frac{-\log_\beta
|I^P_n(\beta)|}{n}.
\end{equation}

We will prove that for all $\beta>1$, the function $\underline{D}(\beta)$ is constantly 1, while for $\overline{D}(\beta)$, we get that $\overline{D}(\beta)\leq 1+\tau(\beta)$, where $\tau(\beta)$ is a constant depending on $\beta$ and $\tau(\beta)\leq 1$, for more details see Lemma \ref{EVA}. Such result is somewhat similar to that on given rate of the growth of the length on cylinders containing $x\in[0,1)$ which is studied by A. Fan and B. Wang \cite{AB}. Thus, analogously, for every $1<\delta \leq 2$, we define the \emph{$\delta$-irregular set} in parameter space as $$F_\delta^P=\{\beta \in (1,+\infty): \overline{D}(\beta)=\delta\}.$$ It is of interest to know how large the $\delta$-irregular set is. Inspired by this, we give the Hausdorff dimension of the $\delta$-irregular set $F_\delta^P$ as the following theorem and we also can get the Lebesgue measure of $F_\delta^P$ after that.

\begin{theorem}\label{dim}
Let $1<\delta\leq 2$. We have $$\dim_{\rm H}F_\delta^P=\frac{2-\delta}{\delta}$$ where $\dim_{\rm H}$ denotes the Hausdorff dimension.
\end{theorem}

The above result implies that $\dim_{\rm H}F_\delta^P=\frac{2-\delta}{\delta}<1$ for all $1<\delta\leq 2$. Therefore, we immediately have the following corollary.
\begin{corollary}
Let $1<\delta\leq 2$. Then $\mathcal{L}(F_\delta^P)=0$ where $\mathcal{L}$ denote the Lebesgue measure.
\end{corollary}

We finish this introduction by illustrating the organization of this paper. In the next section, we will devote to reviewing some of standard fact on the properties of $\beta$-expansion. In Section 3, we will estimate the lengths of $n$th cylinders of a fixed $\beta>1$ in the parameter space $\{\beta\in \R: \beta>1\}$. We are going to give the estimation of lower and upper density of $\beta$ in this section as well. In this section, the $n$th recurrence time of $\beta$ will play an important role. Finally, we will prove Theorem \ref{dim} in Section 4. For more details about cylinders in parameter space, we refer the reader to \cite{LWX,PS}. For more dimensional results concerning the irregular sets, see \cite{AB,BWJ,TD} and references therein.

\section{Preliminaries}

In this section, we will give some basic facts of $\beta$-expansions and fix some notations. For more properties of
$\beta$-expansions see the paper of Parry \cite{P}, R\'{e}nyi \cite{R} and the references therein.

Recall the $\beta$-expansion of $1$ is given by $$1=\frac{\varepsilon_1(1,\beta)}{\beta}+\cdots+\frac{\varepsilon_n(1,\beta)}{\beta^n}+\cdots.$$ where $\varepsilon_n(1,\beta)=\lfloor \beta T_{\beta}^{n-1}1 \rfloor$. If above series is finite,  i.e.,\ $\varepsilon(1,\beta)$ ends with $0^{\infty}=00\cdots,$ then $\beta$ is called \emph{a simple Parry number}. In this case, define $$\varepsilon^\ast(1,\beta):=(\varepsilon^\ast_1(1,\beta),\varepsilon^\ast_2(1,\beta),\ldots)=(\varepsilon_1(1,\beta),
\ldots,\varepsilon_{m-1}(1,\beta),\varepsilon_m(1,\beta)-1)^{\infty},$$ where $\varepsilon_m(1,\beta)$ is the last digit in $\varepsilon(1,\beta)$  that is not equal to $0$ and $\omega^\infty$ denotes the periodic sequence $(\omega,\omega,\omega,\ldots)$. Otherwise, then $$\varepsilon^\ast(1,\beta):=(\varepsilon^\ast_1(1,\beta),\varepsilon^\ast_2(1,\beta),\ldots)=(\varepsilon_1(1,\beta),
\varepsilon_2(1,\beta),\ldots).$$ In both case, the sequence $\varepsilon^\ast(1,\beta)$ is said to be \emph{the infinite $\beta$-expansion of $1$} and we always have that $$1=\frac{\varepsilon^\ast_1(1,\beta)}{\beta}+\cdots+\frac{\varepsilon^\ast_n(1,\beta)}{\beta^n}+\cdots.$$

From the definition of $T_\beta$, it is obvious that, for an integer $n \geq 1$ and $x \in [0,1]$, the $n$th digit $\varepsilon_n(x,\beta)$ of $x$ belongs to the alphabet $\A=\{0,\ldots,\lfloor \beta \rfloor-1\}$ when $\beta$ is an integer and $\A=\{0,\ldots,\lfloor \beta \rfloor\}$ otherwise. What we should note here is that not all sequences $\omega \in \A^\N$ are the $\beta$-expansion of $x$. This leads to the notation of \emph{$\beta$-admissible sequence}.

A word $(\varepsilon_1,\ldots,\varepsilon_n)$ is called \emph{admissible} with respect to the base $\beta$ if there
exists an $x \in [0, 1)$ such that the $\beta$-expansion of $x$ satisfies $\varepsilon_1(x,\beta)=\varepsilon_1,\ldots,\varepsilon_n(x,\beta)= \varepsilon_n$.

We write $\Sigma_\beta^n$ as the collection of all $\beta$-admissible sequences
with length $n$, i.e.,\ $$\Sigma_\beta^n=\{(\varepsilon_1, \varepsilon_2,\ldots,\varepsilon_n)\in \A^n:
\exists\ x \in (0,1),\ {\rm such\ that\ }\varepsilon_j(x,\beta)=\varepsilon_j, \forall\ 1
\leq j \leq n\}.$$

Now we give the definition of \emph{lexicographical order $<_{\rm{lex}}$} in the symbolic space $\A^\N$ as
follows: $$(\varepsilon_1, \varepsilon_2,\ldots)<_{\rm{lex}}(\varepsilon'_1, \varepsilon'_2,\ldots)$$if there exists an integer $k \geq 1$ such that, for all $1 \leq j < k$, $\varepsilon_j=\varepsilon'_j$ but $\varepsilon_k<\varepsilon'_k$. The symbol $\leq_{\rm{lex}}$ means $=$ or $<_{\rm{lex}}$. Moreover, we extend the lexicographical order to words: let $m,n\geq 1$, for two words $\omega=(\omega_1,\ldots,\omega_n)$, $\omega'=(\omega'_1,\ldots,\omega'_m)$ of nonnegative integers,
$$\omega <_{\rm{lex}} \omega' \Longleftrightarrow (\omega_1,\ldots,\omega_n,0^\infty)<_{\rm{lex}}(\omega'_1,\ldots,\omega'_m,0^\infty).$$

The characterization and properties of the admissibility of a sequence which relies heavily on the infinite $\beta$-expansion of $1$ are given by Parry \cite{P} as the following theorem.
\begin{theorem}[Parry \cite{P}]\label{P1}
Given $\beta > 1$, the followings hold:

(1) For every $n\geq 1$,$$\omega=(\omega_1,\ldots,\omega_n)\in \Sigma_\beta^n \Longleftrightarrow \sigma ^i\omega \leq_{\rm{lex}} (\varepsilon_1^\ast(1,\beta),\ldots,\varepsilon_{n-i}^\ast(1,\beta))\ for\ all\ i \geq 1,$$ where $\sigma$ is the shift operator such that $\sigma\omega=(\omega_2,\omega_3,\ldots).$

(2) For every $i\geq 1$, $\sigma ^i\varepsilon^\ast(1,\beta) \leq_{\rm{lex}} \varepsilon^\ast(1,\beta).$

(3) The function $\beta \mapsto \varepsilon^\ast(1,\beta)$ is strictly increasing with respect to $\beta > 1$. Therefore, if $1<\beta_1<\beta_2$, for any $n \geq 1$, we have $$\Sigma_{\beta_1}^n \subset \Sigma_{\beta_2}^n.$$
\end{theorem}

From the above result, we can see that the $\beta$-expansion of the unit $1$ serves as an important role in $\beta$-expansion. Consequently, we now give a characterization of $\beta$-expansion of the unit $1$ for some $\beta$ which was introduced by Parry \cite{P}. Before doing this, let us first give the definition of a \emph{self-admissible word} and \emph{self-admissible sequence}.
\begin{definition}
A word $\omega=(\omega_1, \ldots,\omega_n)$ is said to be self-admissible if for every $1\leq i<n$,$$ \sigma ^i\omega \leq_{\rm{lex}} (\omega_1,\ldots,\omega_{n-i}). $$
An infinite digit sequence $\omega=(\omega_1, \omega_2,\ldots)$ is called self-admissible if $ \sigma ^i\omega \leq_{\rm{lex}}\omega$ for each $i\geq 1$.
\end{definition}

For the sake of convenience, we denote $\Lambda_n$ as the collection of all self-admissible words with length $n$, i.e.,\ $$\Lambda_n=\{\omega=(\omega_1, \omega_2,\ldots,\omega_n):\ {\rm for\ every}\ 1\leq i<n,\ \sigma ^i\omega \leq_{\rm{lex}} (\omega_1,\ldots,\omega_{n-i}) \}.$$

Now we give the characterization of the $\beta$-expansion of $1$ as follows.
\begin{theorem}[Parry \cite{P}]\label{P2}
A finite word or an infinite sequence $(\varepsilon_1, \varepsilon_2,\ldots,\varepsilon_n,\ldots)$ is the expansion of $1$ for some $\beta > 1$ if and only if it is self-admissible.
\end{theorem}

The following result given by R\'{e}nyi shows that dynamical system $([0,1],\ T_\beta)$ admits $\log \beta$ as its topological entropy. The symbol $\sharp$ represents the number of elements of a finite set in the rest of this paper.
\begin{theorem}[R\'{e}nyi \cite{R}]\label{cardinity}
Given $\beta >1$. For any $n \geq 1$,$$\beta^n \leq \sharp \Sigma_\beta^n \leq
\beta^{n+1}/(\beta-1).$$
\end{theorem}
\section{Cylinders of $\beta$ in the parameter space}
From this section to the end of this paper, we consider the $n$th cylinders of some fixed $\beta>1$ in parameter space $(1,+\infty)$. We will give an evaluation of the lower and upper bound of the length of the $n$th cylinder in the parameter space of $\beta$, which is closely related to the notion of $n$th recurrence time of $\beta$.
\begin{definition}
Let $n\geq 1$ and $\omega=(\omega_1, \ldots,\omega_n)\in \Lambda_n$, we can define the cylinder of $\omega$ in the parameter space as $$I^P_n(\omega)=\{\beta\in(1,+\infty):\varepsilon_1(1,\beta)=\omega_1,\ldots, \varepsilon_n(1,\beta)=\omega_n\},$$
i.e.,\ the set of $\beta$ whose $\beta$-expasion of $1$ begins with the common prefix $\omega_1, \cdots,\omega_n$. The length of the cylinder in the parameter space is denoted as $|I^P_n(\omega)|$. For convenience, we also denote the left endpoint of $I^P_n(\omega)$ as $\underline{\beta}(\omega)$ and the right endpoint of $I^P_n(\omega)$ as $\overline{\beta}(\omega)$.
\end{definition}

\subsection{Recurrence time of $\beta$}
\begin{definition}
Let $\omega=(\omega_1,\ldots,\omega_n)\in \Lambda_n$, we can define the recurrence time $\tau(\omega)$ of $\omega$ as
\begin{equation}\label{tau}
\tau(\omega):=\inf\{1\leq k < n:\sigma^k(\omega_1,\ldots,\omega_n)=(\omega_1,\ldots, \omega_{n-k})\}.
\end{equation}
If such an integer $k$ does not exist, we set $\tau(\omega)=n$, and the word $\omega$ is called non-recurrent word.
\end{definition}

The following result is immediate by the above definition.

\begin{remark}
(1) Denote
\begin{equation}\label{t}
t(\omega):= n-\left \lfloor\frac{n}{\tau(\omega)}\right \rfloor\tau(\omega).
\end{equation}
From the definition of the recurrence time $\tau(\omega)$, we can easily obtain that:
$$(\omega_1,\ldots,\omega_n)=\left((\omega_1,\ldots,\omega_{\tau(\omega)})^{\lfloor \frac{n}{\tau(\omega)}\rfloor},\omega_1,\ldots,\omega_{t(\omega)}\right),$$ where $\omega^k=(\underbrace{\omega,\ldots,\omega}_k)$ for every $k\geq 1$.

(2) If $\omega=(\omega_1,\ldots,\omega_n)$ is non-recurrent, then the word $(\omega_1,\ldots,\omega_n,0^k)$ is non-recurrent for all $k\geq 1$.
\end{remark}

Now we give a property of the recurrence time as follow.
\begin{lemma}\label{re}
Let $\omega=(\omega_1,\ldots,\omega_n)\in \Lambda_n$, if $\tau(\omega)=k$, then the word $(\omega_1,\ldots,\omega_k)$ is non-recurrent.
\end{lemma}
\begin{proof}
When $k=n$, we easily get that the result is ture.

When $k<n$, then the definition of $\tau(\omega)$ provides that for all $1\leq i < k$, we have $$\sigma^i\omega=(\omega_{i+1},\ldots,\omega_k,(\omega_1,\ldots,\omega_k)^{\lfloor\frac{n}{k}\rfloor-1},\omega_1,\ldots,\omega_{n-\lfloor\frac{n}{k}\rfloor k})<_{\rm lex}(\omega_1,\ldots,\omega_{k-i},\omega_{k-i+1},\ldots,\omega_{n-i}).$$ Moreover, the self-admissibility of $\omega$ gives that $$((\omega_1,\ldots,\omega_k)^{\lfloor\frac{n}{k}\rfloor-1},\omega_1,\ldots,\omega_{n-\lfloor\frac{n}{k}\rfloor k})\geq_{\rm lex}(\omega_{k-i+1},\ldots,\omega_{n-i}).$$ Thus, we have $(\omega_{i+1},\ldots,\omega_k)<_{\rm lex}(\omega_1,\ldots,\omega_{k-i})$ for each $1\leq i <k$, that is, $(\omega_1,\ldots,\omega_k)$ is non-recurrent.
\end{proof}

We write $\omega^+$ as the next word of $\omega$ in $\Lambda_n$ in the sense of lexicographical order, the definition of recurrence time and the criterion of self-admissibility of a sequence give the following result.

\begin{lemma}\label{nextword}
(1) Let $\omega=(\omega_1,\ldots,\omega_n)\in \Lambda_n$ be non-recurrent, then $\omega^+=(\omega_1,\ldots,\omega_n+1).$

(2) Let $\omega=(\omega_1,\ldots,\omega_n)\in \Lambda_n$ with $\tau(\omega)=k<n$, then $\omega^+=(\omega_1,\ldots,\omega_k+1,0^{n-k}).$
\end{lemma}
\begin{proof}
(1) It is obvious that for any word $\omega' \in  \Lambda_n$ with $\omega <_{\rm{lex}} \omega'$, we have $(\omega_1,\ldots,\omega_n+1) \leq_{\rm{lex}} \omega'$. Thus, we only need to show that $(\omega_1,\ldots,\omega_n+1)$ is self-admissible.

In fact, since $\omega$ is a non-recurrent word, then for every $1 \leq k < n$,  we have $$\sigma^k(\omega_1,\ldots,\omega_n)\neq(\omega_1,\ldots, \omega_{n-k}),$$ moreover, the self-admissibility of $\omega$ ensures that $$\sigma^k(\omega_1,\ldots,\omega_n)\leq_{\rm{lex}}(\omega_1,\ldots, \omega_{n-k}),$$ so $\sigma^k(\omega_1,\ldots,\omega_n)<_{\rm{lex}}(\omega_1,\ldots, \omega_{n-k})$ which implies $\sigma^k(\omega_1,\ldots,\omega_n+1)\leq_{\rm{lex}}(\omega_1,\ldots, \omega_{n-k})$.

(2) The recurrence time $\tau(\omega)=k<n$ of $\omega$ implies that $(\omega_1,\ldots,\omega_k)$ is non-recurrent (Lemma \ref{re}), then by (1), we get that $(\omega_1,\ldots,\omega_k+1)$ is self-admissible, and $(\omega_1,\ldots,\omega_k+1,0^{n-k})$ is self-admissible as well.

We claim that for every $\omega' \in  \Lambda_n$ with $\omega <_{\rm{lex}} \omega'$, we have $(\omega_1,\ldots,\omega_k+1,0^{n-k}) \leq_{\rm{lex}} \omega'$. In fact, assume that the claim is false, then there exists $u \in  \Lambda_n$, such that $$\omega <_{\rm{lex}} u <_{\rm{lex}}(\omega_1,\ldots,\omega_k+1,0^{n-k}).$$ Write $u=(\omega'_1,\ldots,\omega'_n)$, then there exists an integer $k_0$ with $k <k_0\leq n$, such that $\omega_j=\omega'_j$ for every $0 \leq j <k_0$  but $\omega_{k_0}<\omega'_{k_0}$. Let $k_0=\ell k+t$, where $\ell \geq 1$, $0 < t \leq k$, the definition of $\tau(\omega)=k$ implies that $$\sigma^{\ell k}(\omega_1,\ldots,\omega_{k_0}) = (\omega_{\ell k+1},\ldots,\omega_{\ell k+ t-1},\omega_{\ell k+t})=(\omega_1,\ldots,\omega_{t-1},\omega_t) = (\omega_1,\ldots,\omega_{t-1},\omega_{k_0}),$$ and $$\sigma^{\ell k}(\omega'_1,\ldots,\omega'_{k_0}) =(\omega'_{\ell k+1},\ldots,\omega'_{\ell k+ t-1},\omega'_{k_0})= (\omega_{\ell k+1},\ldots,\omega_{\ell k+ t-1},\omega'_{k_0}) =(\omega_1,\ldots,\omega_{t-1},\omega'_{k_0}).$$ The self-admissibility of  $\omega'_0 $ shows that $$\sigma^{\ell k}(\omega'_1,\ldots,\omega'_{k_0})=(\omega_1,\ldots,\omega_{t-1},\omega'_{k_0})\leq_{\rm{lex}} (\omega'_1,\ldots,\omega'_t)=(\omega_1,\ldots,\omega_t) = (\omega_1,\ldots,\omega_{t-1},\omega_{k_0}).$$ The second equality is ensured by $t\leq k <k_0$. A contraction with $\omega_{k_0}<\omega'_{k_0}$.
\end{proof}

Applying the definition of the recurrence time of $\omega\in \Lambda_n$, the \emph{$n$th recurrence time of $\beta$} can be defined as  \begin{equation}\label{taun}
\tau_n(\beta):=\tau(\varepsilon_1(1,\beta),\ldots,\varepsilon_n(1,\beta)).
\end{equation}
Now we give a simple fact of the $n$th recurrence time $\tau_n(\beta)$.
\begin{lemma}\label{rec}
Let $\beta>1$, then we have $\tau_n(\beta)$ is increasing to $\infty$ as $n\rightarrow\infty$.
\end{lemma}
\begin{proof}
On the one hand, we show that  $\tau_n(\beta)$ is increasing as $n$ increase. If there exists $n_0\geq 1$ such that $\tau_{n_0+1}(\beta)<\tau_{n_0}(\beta)$, by the definition of recurrence time, we have $$\sigma^{\tau_{n_0+1}(\beta)}(\varepsilon_1(1,\beta),\ldots,\varepsilon_{n_0+1}(1,\beta))=\left(\varepsilon_1(1,\beta),\ldots, \varepsilon_{n_0+1-\tau_{n_0+1}(\beta)}(1,\beta)\right).$$ This means
$$\sigma^{\tau_{n_0+1}(\beta)}(\varepsilon_1(1,\beta),\ldots,\varepsilon_{n_0}(1,\beta))=\left(\varepsilon_1(1,\beta),\ldots, \varepsilon_{n_0-\tau_{n_0+1}(\beta)}(1,\beta)\right).$$  Which contradicts the definition of $\tau_{n_0}(\beta)$.

On the other hand, we prove that $\tau_n(\beta)\rightarrow \infty$  as $n\rightarrow\infty$. If $\tau_n(\beta)<\infty$, then there exists $M\geq1$ such that $\tau_n\leq M$ for every $n\geq1$. Since $\tau_n(\beta)$ is increasing by the above discussion, so the maximum of $\tau_n(\beta)$ exists, denote $\max\{\tau_n(\beta)\}=N$. This illustrates that the infinite $\beta$-expansion of $1$ equals to $(\varepsilon_1(1,\beta),\ldots,\varepsilon_N(1,\beta))^\infty$ in other words, ones get that $$\varepsilon(1,\beta)=(\varepsilon_1(1,\beta),\ldots,\varepsilon_N(1,\beta)+1,0^\infty).$$
A contradiction.
\end{proof}

In addition, a sufficient condition to ensure that a word is non-recurrent is given by B. Li et al.\ \cite{LP} which will be of importance to estimate the lower density of $\beta$ in parameter space.
\begin{lemma}\label{non}
Suppose that $(\omega_1,\ldots,\omega_{n-1},\omega_n)$ and $(\omega_1,\ldots,\omega_{n-1},\omega'_n)$ are both self-admissible and $0\leq\omega_n<\omega'_n$. Then $\tau(\omega_1,\ldots,\omega_n)=n.$
\end{lemma}

\subsection{The lengths of the $n$th cylinder of $\beta$ in the parameter space  }

We first recall a result of  Schemling \cite{S}(see also \cite{LP,LW}) which gives the concrete expression of the left endpoints $\underline{\beta}(\omega)$ and right endpoints $\overline{\beta}(\omega)$ of the cylinder  $I^P_n(\omega)$ and gives an evaluation of the upper bound of $|I^P_n(\omega)|$.
\begin{lemma}[Schemling\cite{S}]\label{S}
Let $\omega=(\omega_1,\ldots,\omega_n)\in \Lambda_n$, $\underline{\beta}(\omega)$ is given as the only solution in $[1,+\infty)$ of the equation
$$1=\frac{\omega_1}{x}+\cdots+\frac{\omega_n}{x^n}.$$
$\overline{\beta}(\omega)$ is the limit of the unique solution $\{\omega_{n+k}\}_{k \geq 0}$ in $(1,+\infty)$ of the equation $$1=\frac{\omega_1}{x}+\cdots+\frac{\omega_n}{x^n}+\frac{\omega_{n+1}}{x^{n+1}}+\cdots+\frac{\omega_{n+k}}{x^{n+k}},\ k \geq 0,$$
where $(\omega_1,\ldots,\omega_n,\omega_{n+1},\ldots,\omega_{n+k})$ is the maximal word of length $n+k$ beginning with $(\omega_1,\ldots,\omega_n)$ in the sense of lexicographical order. If $\underline{\beta}(\omega)>1$, then the set $I^P_n(\omega)$ is a half open interval $[\underline{\beta}(\omega),\overline{\beta}(\omega));$ if $\underline{\beta}(\omega)=1$, then $I^P_n(\beta)$ is an open interval $(1,\overline{\beta}(\omega)).$
Moreover, $$|I^P_n(\omega)|\leq \overline{\beta}(\omega)^{-n+1}.$$
\end{lemma}

The following result introduced by B. Li et al.\ \cite{LP} and Lemma \ref{nextword} give the distribution of $I^P_n(\omega)$ (see Corollary \ref{dis}) for every $\omega \in \Lambda_n$ which implies that $I^P_n(\omega)\cap I^P_n(\omega')=\emptyset$ for different $\omega,\omega' \in \Lambda_n$.
\begin{lemma}[B. Li et al.\ \cite{LP}]\label{max}
Let $\omega=(\omega_1,\ldots,\omega_n)\in \Lambda_n$ with $\tau(\omega)=k$. Then the $\overline{\beta}(\omega)$-expansion of $1$ is given  as $$\varepsilon(1,\overline{\beta}(\omega))= (\omega_1,\ldots,\omega_{k}+1,0^\infty).$$
\end{lemma}
\begin{corollary}\label{dis}
Let $\omega=(\omega_1,\ldots,\omega_n)\in \Lambda_n$. Write $\omega^+$ as the next word of $\omega$ in $\Lambda_n$, then we have $\underline{\beta}(\omega^+)=\overline{\beta}(\omega)$.
\end{corollary}
\begin{proof}
It is a direct result from Lemma \ref{nextword} and Lemma \ref{max} .
\end{proof}

Now we discuss the property of the solutions of a self-admissible word as the following result. For any $n \geq 1$ and self-admissible word $\omega=(\omega_1,\ldots,\omega_n)\in \Lambda_n$, let $\beta(\omega) \geq 1$ be the unique positive solution of the equation $1=\sum\limits^n_{i=1}\frac{\omega_i}{x^i}.$ We soon establish the following lemma.
\begin{lemma}\label{SO}
Let $\omega=(\omega_1,\ldots,\omega_n)$ and  $\omega'=(\omega'_1,\ldots,\omega'_n)$ be self-admissible, if $\omega <_{\rm{lex}}\omega'$, then $\beta(\omega)< \beta(\omega')$.
\end{lemma}
\begin{proof}
If $\omega <_{\rm{lex}}\omega'$, then
$$1 = \frac{\omega_1}{\beta(\omega)}+\cdots+\frac{\omega_n}{\beta(\omega)^n} =\frac{\omega'_1}{\beta(\omega')}+\cdots+\frac{\omega'_n}{\beta(\omega')^n} >\frac{\omega_1}{\beta(\omega')}+\cdots+ \frac{\omega_n}{\beta(\omega')^n}.$$
Therefore, $\beta(\omega)< \beta(\omega')$.
\end{proof}

From what has been discussed above, we now give the proof of Theorem \ref{LIMIT}.\\
\textbf{Proof of Theorem \ref{LIMIT}}
(1) We first prove that $\underline{\beta}_n$ is increasing as $n$ increases and $\lim\limits_{n\rightarrow\infty}\underline{\beta}_n=\beta.$ By Lemma \ref{S}, we have  $\underline{\beta}_n$ is the only solution in $[1,+\infty)$ of the equation
$$1=\frac{\varepsilon_1(1,\beta)}{x}+\cdots+\frac{\varepsilon_n(1,\beta)}{x^n}.$$ Since $(\varepsilon_1(1,\beta),\cdots,\varepsilon_n(1,\beta))$ is increasing as $n$ increases, therefore, we obtain that $\underline{\beta}_n$ is increasing as $n \rightarrow \infty$ by Lemma \ref{SO}, and since $$1=\frac{\varepsilon_1(1,\beta)}{\beta}+\cdots+\frac{\varepsilon_n(1,\beta)}{\beta^n}+\cdots,$$
we have $\lim\limits_{n\rightarrow\infty}\underline{\beta}_n=\beta.$

(2) Now we show that $\overline{\beta}_n$ is decreasing when $n$ increases and $\lim\limits_{n\rightarrow\infty}\overline{\beta}_n=\beta.$ By Lemma \ref{max}, we have $$\varepsilon(1,\overline{\beta}_n)=(\varepsilon_1(1,\beta),\ldots,\varepsilon_{\tau_n}(1,\beta)+1,0^\infty)$$  and $$\varepsilon(1,\overline{\beta}_{n+1})=(\varepsilon_1(1,\beta),\ldots,\varepsilon_{\tau_{n+1}}(1,\beta)+1,0^\infty).$$ Besides, Lemma \ref{rec} gives $\tau_n(\beta)\leq\tau_{n+1}(\beta)$, that is, $\varepsilon_{\tau_n(\beta)}(1,\overline{\beta}_n)> \varepsilon_{\tau_n(\beta)}(1,\overline{\beta}_{n+1})$ when $\tau_n(\beta)<\tau_{n+1}(\beta)$ and $\varepsilon_{\tau_n(\beta)}(1,\overline{\beta}_n)= \varepsilon_{\tau_n(\beta)}(1,\overline{\beta}_{n+1})$ when $\tau_n(\beta)=\tau_{n+1}(\beta)$. So we have $\varepsilon(1,\overline{\beta}_n)\geq \varepsilon(1,\overline{\beta}_{n+1})$. It follows from Lemma \ref{SO} that $\overline{\beta}_n \geq \overline{\beta}_{n+1}$, that is, $\overline{\beta}_n$ is decreasing when $n$ increases.

Note that $|I^P_n(\beta)|\leq \overline{\beta}_n^{-n+1}\rightarrow 0$ as $n\rightarrow +\infty$, so $$\lim\limits_{n\rightarrow\infty}\overline{\beta}_n=\lim\limits_{n\rightarrow\infty}\underline{\beta}_n=\beta.$$
$\hfill\Box$

Fix $\beta>1$. Let
\begin{equation}\label{tn}
t_n(\beta):= n-\left \lfloor\frac{n}{\tau_n(\beta)}\right \rfloor\tau_n(\beta).
\end{equation}
where $\tau_n(\beta)$ is the $n$th recurrence time of $\beta$ defined as (\ref{taun}). The length of $n$th cylinders in parameter space is also considered by B. Li et al.\ \cite{LP}, the following result provides the estimation of the the lower bound of $|I^P_n(\beta)|$.
\begin{lemma}[B. Li et al.\ \cite{LP}]\label{EST}
Let $\beta>1$. Then, we have
$$|I^P_n(\beta)|\geq \left\{
\begin{aligned}
C_n(\beta)\overline{\beta}_n^{-n}\ & , & when\ t_n(\beta)=0; \\
C_n(\beta)\overline{\beta}_n^{-n}\left(\frac{\varepsilon_{t_n(\beta)+1}(1,\beta)}{\overline{\beta}_n}+ \cdots+\frac{\varepsilon_{\tau_n(\beta)}(1,\beta)+1}{\overline{\beta}_n^{\tau_n(\beta)-t_n(\beta)}}\right) & , & otherwise,
\end{aligned}
\right.$$where
\begin{equation}\label{cn}
C_n(\beta):=\frac{(\underline{\beta}_n-1)^2}{\underline{\beta}_n}.
\end{equation}
\end{lemma}

\subsection{The upper and lower density in parameter space}
Recall the definition of $\tau_n(\beta)$(\ref{taun}) and $t_n(\beta)$(\ref{tn}). We denote $\tau(\beta):= \limsup\limits_{n\rightarrow \infty}\frac{\tau_n(\beta)-t_n(\beta)}{n}.$ Now we give the evaluation of $\overline{D}(\beta)$ and $\underline{D}(\beta)$ as follows,
\begin{lemma}\label{EVA}
Let $\beta>1$ be a real number. Then,
 $$\underline{D}(\beta)=1$$
 and
 $$\overline{D}(\beta)\leq 1+\tau(\beta).$$
\end{lemma}
\begin{proof}(1) By Lemma \ref{S}, we have $|I^P_n(\beta)|\leq \overline{\beta}_n^{-n+1}$ for all $n\geq 1$. Consequently,  $\underline{D}(\beta)\geq 1$ by Theorem \ref{LIMIT}.

Now we find a sequence $\{n_k\}_{k\geq1}$ such that $\underline{D}(\beta)=1$. Since $\tau_n(\beta)\rightarrow\infty$ as $n\rightarrow \infty$ (Lemma \ref{rec}). We can choose an increasing sequence $\{n_k\}_{k\geq1}$ such that $\tau_{n_k}=n_k$.

In fact, assume that $\varepsilon(1,\beta)=(\varepsilon_1,\ldots,\varepsilon_n,\ldots)$, let $n_1=1$, and after that we recursively choose $$n_k=\inf \{n: \tau(\varepsilon_1,\ldots,\varepsilon_n)>\tau_{n_{k-1}}\}.$$ Then for every $k\geq 1$, let $n_k=\ell_k \tau_{n_{k-1}}+t_{n_k}$ where $0<t_{n_k}\leq \tau_{n_{k-1}}$, so by the definition of $\tau_{n_k}$, we have
\begin{equation}\label{e1}
(\varepsilon_1,\ldots,\varepsilon_{n_k-1})= \left((\varepsilon_1,\ldots,\varepsilon_{\tau_{n_{k-1}}})^{\ell_k},\varepsilon_1,\ldots,\varepsilon_{t_{n_k}-1}\right).
\end{equation}
Furthermore, the self-admissibility of $\varepsilon(1,\beta)$ gives $$(\varepsilon_1,\ldots,\varepsilon_{n_k})\leq_{\rm lex} \left((\varepsilon_1,\ldots,\varepsilon_{\tau_{n_{k-1}}})^{\ell_k},\varepsilon_1,\ldots,\varepsilon_{t_{n_k}-1},\varepsilon_{t_{n_k}}\right).$$
But the definition of $n_k$ provides that $$(\varepsilon_1,\ldots,\varepsilon_{n_k})\neq \left((\varepsilon_1,\ldots,\varepsilon_{\tau_{n_{k-1}}})^{\ell_k},\varepsilon_1,\ldots,\varepsilon_{t_{n_k}-1},\varepsilon_{t_{n_k}}\right).$$
These mean that
\begin{equation}\label{e2}
(\varepsilon_1,\ldots,\varepsilon_{n_k})<_{\rm lex} \left((\varepsilon_1,\ldots,\varepsilon_{\tau_{n_{k-1}}})^{\ell_k},\varepsilon_1,\ldots,\varepsilon_{t_{n_k}-1},\varepsilon_{t_{n_k}}\right).
\end{equation}
Combining (\ref{e1}) and (\ref{e2}), ones can get from  Lemma \ref{non} that $\tau_{n_k}=n_k$. So Lemma \ref{EST} gives $|I^P_{n_k}(\beta)|\geq C_{n_k}(\beta)\overline{\beta}_{n_k}^{-{n_k}}$, where $C_n(\beta)$ is given as (\ref{cn}), that is,
$$\begin{array}{rcl}
\underline{D}(\beta) &\leq& \lim\limits_{k\rightarrow \infty}\frac{-\log_\beta
|I^P_{n_k}(\beta)|}{n_k} \\
&\leq & \lim\limits_{k\rightarrow \infty}\frac{-\log_\beta
(C_{n_k}(\beta)\overline{\beta}_{n_k}^{-{n_k}})}{n_k}\\&=&1.
\end{array} $$
The last equality is given by Theorem \ref{LIMIT}. Thus, for every $\beta>1$, $\underline{D}(\beta)=1$.

(2) By  Lemma \ref{EST}, we have $$|I^P_n(\beta)| \geq C_n(\beta) \overline{\beta}_n^{-n} \left (\frac{\varepsilon_{t_n(\beta)+1}(1,\beta)}{\overline{\beta}_n} + \cdots + \frac{\varepsilon_{\tau_n(\beta)}(1,\beta)+1}{\overline{\beta}_n^{\tau_n(\beta)-t_n(\beta)}} \right),$$ where $C_n(\beta)$ is defined as (\ref{cn}). Thus, we deduce that, for each $n\geq 1$,
$$|I^P_n(\beta)|\geq C_n(\beta) \overline{\beta}_n^{-n-\tau_n(\beta)+t_n(\beta)}.$$
Therefore, $\overline{D}(\beta)\leq 1+\tau(\beta)$ by Theorem \ref{LIMIT}.
\end{proof}

The following lemma gives an example of $\beta$ which has the property of $\underline{D}(\beta)=\overline{D}(\beta)=1.$
\begin{lemma}
If $\beta$ satisfies that there exits $k\geq 1$ such that for each $n\geq k$, $\tau_n(\beta)=n$, then we have $\underline{D}(\beta)=\overline{D}(\beta)=1.$
\end{lemma}
\begin{proof}
Since for all $n\geq k$, we have $\tau_n(\beta)=n$. It therefore follows from Lemma \ref{S} and Lemma \ref{EST} that
$$\frac{(\underline{\beta}_n-1)^2}{\underline{\beta}_n}\overline{\beta}_n^{-n} \leq |I_n^P(\beta)|\leq \overline{\beta}_n^{-n+1}, $$
for every $n\geq k$ which implies that $\underline{D}(\beta)=\overline{D}(\beta)=1$ by Theorem \ref{LIMIT}.
\end{proof}

\section{Dimension of $\delta$-irregular cylinders}

\subsection{Evaluations of the cardinality of some related sets}
We first introduce some notations which play an important role in estimating upper bound of $\dim_{\rm{H}}I_{\delta}^P$.

Let $\beta_2>\beta_1>1$, for every $n\geq1$, we write $$\Lambda_n(\beta_1,\beta_2):=\{\omega=(\omega_1,\ldots,\omega_n)\in \Lambda_n: \exists\ \beta \in (\beta_1,\beta_2],\ s.t.\ \varepsilon_1(1,\beta)=\omega_1, \ldots, \varepsilon_n(1,\beta)=\omega_n\}.$$ Fixed $1\leq k <n$, for every $\omega\in \Lambda_n(\beta_1,\beta_2)$, let $t(\omega)$ be defined as (\ref{t}). Define
\begin{equation}\label{nk}
\Lambda_{n,k}(\beta_1,\beta_2):= \{\omega=(\omega_1,\ldots,\omega_n)\in \Lambda_n(\beta_1,\beta_2): \omega_{t(\omega)+1}= \cdots= \omega_{t(\omega)+k}=0\}.
\end{equation}

Now we provide an evaluation of the cardinality of the set $\Lambda_{n,k}(\beta_1,\beta_2)$.
\begin{lemma}\label{cardinality}
Let  $\Lambda_{n,k}(\beta_1,\beta_2)$ be a nonempty set defined above. Then we have $$\sharp \Lambda_{n,k}(\beta_1,\beta_2)\leq \frac{\beta_2^{n-k+1}}{\beta_2-1}.$$
\end{lemma}

\begin{proof}
Let $T:\Lambda_{n,k}(\beta_1,\beta_2)\rightarrow \Sigma_{\beta_2}^{n-k}$ where $$T(\omega_1,\ldots,\omega_{t(\omega)},0^k,\omega_{t(\omega)+k+1},\ldots,\omega_n)=(0^{t(\omega)},\omega_{t(\omega)+k+1},\ldots,\omega_n).$$ By the self-admissibility of $(\omega_1,\ldots,\omega_{t(\omega)},0^k,\omega_{t(\omega)+k+1},\ldots,\omega_n)$ and the monotonicity of $\varepsilon(1,\beta)$ with respect to $\beta$ (Theorem \ref{P1}(3)), we have $(0^{t(\omega)},\omega_{t(\omega)+k+1},\ldots,\omega_n)\in \Sigma_{\beta_2}^{n-k}$ for every $\omega \in \Lambda_{n,k}(\beta_1,\beta_2).$ Therefore, $T$ is well defined.

Moreover, we claim that $T$ is injective, that is, for each $\omega=(\omega_1,\ldots,\omega_n),\ \omega'=(\omega'_1,\ldots,\omega'_n)\in \Lambda_{n,k}(\beta_1,\beta_2)$, if $T(\omega)=T(\omega')$, we have $\omega=\omega'$. To prove $\omega=\omega'$, we first prove that $t(\omega)=t(\omega'),$ if it is not true, without loss of generality, assume that $t(\omega)>t(\omega')$, write $$\omega=(\omega_1,\ldots,\omega_{t(\omega)},0^k,\omega_{t(\omega)+k+1},\ldots,\omega_n),$$ then on the one hand, $T(\omega)=T(\omega')$ provides that
\begin{equation}\label{ww1}
(\omega'_{n-t(\omega')+1},\ldots,\omega'_n)=(\omega_{n-t(\omega')+1},\ldots,\omega_n).
\end{equation}
On the other hand, by the definition of $t(\omega)$ and $t(\omega')$, we have
\begin{equation}\label{ww2}
(\omega'_{n-t(\omega')+1},\ldots,\omega'_n)=(\omega'_1,\ldots,\omega'_{t(\omega')}),
\end{equation}
 and
\begin{equation}\label{ww3}
 (\omega_{n-t(\omega')+1},\ldots,\omega_n)=(\omega_{t(\omega)-t(\omega')+1},\ldots,\omega_{t(\omega)}).
\end{equation}Combination of (\ref{ww1}),(\ref{ww2}) and (\ref{ww3}) implies that
$$\omega'=(\omega_{t(\omega)-t(\omega')+1},\ldots,\omega_{t(\omega)},0^k,0^{t(\omega)-t(\omega')},\omega_{t(\omega)+k+1},\ldots,\omega_n).$$

Now we observe $\omega'$, by the self-admissibility of $\omega'$, we get that $$(\omega_1,\ldots,\omega_{t(\omega)})\leq_{\rm lex}(\omega_{t(\omega)-t(\omega')+1},\ldots,\omega_{t(\omega)},0^{t(\omega)-t(\omega')}),$$ but the self-admissibility of $\omega$ implies that $$(\omega_{t(\omega)-t(\omega')+1},\ldots,\omega_{t(\omega)},0^{t(\omega)-t(\omega')})\leq_{\rm lex} (\omega_1,\ldots,\omega_{t(\omega)}),$$ that is, $$(\omega_1,\ldots,\omega_{t(\omega)})=(\omega_{t(\omega)-t(\omega')+1},\ldots,\omega_{t(\omega)},0^{t(\omega)-t(\omega')}).$$ Thus, ones have $$(\omega_1,\ldots,\omega_{t(\omega')})=(\omega_{t(\omega)-t(\omega')+1},\ldots,\omega_{t(\omega)})$$ and $$(\omega_{t(\omega')+1},\ldots,\omega_{t(\omega)})=0^{t(\omega)-t(\omega')}.$$ If $t(\omega)-t(\omega')\geq t(\omega'),$ i.e.,\ $t(\omega)\geq 2t(\omega')$, then we have $\omega_{t(\omega)-t(\omega')+1}=0$ which contradicts that $\omega_1=\omega_{t(\omega)-t(\omega')+1}\geq 1$, if $t(\omega) < 2t(\omega')$ we get $$(\omega_1,\ldots,\omega_{t(\omega')})=(\omega_{t(\omega)-t(\omega')+1},\ldots,\omega_{t(\omega)})= (\omega_{t(\omega)-t(\omega')+1},\ldots,\omega_{t(\omega')},0^{t(\omega)-t(\omega')}).$$That is, $$(\omega_1,\ldots,\omega_{2t(\omega')-t(\omega)})=(\omega_{t(\omega)-t(\omega')+1},\ldots,\omega_{t(\omega')})$$ and $$(\omega_{2t(\omega')-t(\omega)+1},\ldots,\omega_{t(\omega')})=0^{t(\omega)-t(\omega')}.$$ If $t(\omega)-t(\omega')\geq 2t(\omega')-t(\omega),$ i.e.,\ $2t(\omega) \geq 3t(\omega')$, then we have $\omega_{t(\omega)-t(\omega')+1}=0$ which contradicts that $\omega_1=\omega_{t(\omega)-t(\omega')+1}\geq 1$, if $2t(\omega) < 3t(\omega'),$ we get $$(\omega_1,\ldots,\omega_{2t(\omega')-t(\omega)})=(\omega_{t(\omega)-t(\omega')+1},\ldots,\omega_{t(\omega')})= (\omega_{t(\omega)-t(\omega')+1},\ldots,\omega_{2t(\omega')-t(\omega)},0^{t(\omega)-t(\omega')}).$$
By recursion, the existence of $\omega'$ should be ensured by $nt(\omega)<(n+1)t(\omega')$ for every $n\geq 1$. But the assumption that $t(\omega)>t(\omega')$ indicates that $nt(\omega)\geq (n+1)t(\omega')$ when $n\geq \left\lfloor\frac{t(\omega')}{t(\omega)-t(\omega')}\right\rfloor+1$. A contraction. As a consequence, $t(\omega)=t(\omega').$ Then we can easily get that $\omega=\omega'$ by $T(\omega)=T(\omega')$.
Thus,
$$\sharp \Lambda_{n,k}(\beta_1,\beta_2)\leq \sharp \Sigma_{\beta_2}^{n-k} \leq \frac{\beta_2^{n-k+1}}{\beta_2-1}.$$ where the last inequality follows from Theorem \ref{cardinity}.
\end{proof}
\subsection{Upper bound of $\dim_{\rm{H}}F_{\delta}^P$}
For every $\beta_2>\beta_1>1$ and $1< \delta \leq 2$. Let $F_{\delta}^P(\beta_1,\beta_2):=\{\beta \in (\beta_1,\beta_2]: \overline{D}(\beta)= \delta\}.$ Now we estimate the upper bound of  $\dim_{\rm{H}}F_{\delta}^P(\beta_1,\beta_2)$.

For every $1<\eta <\delta$, we have
$$ F_{\delta}^P(\beta_1,\beta_2)\subseteq \bigcap_{N=1}^\infty \bigcup_{n=N}^\infty \{\beta \in (\beta_1,\beta_2]: |I^P_n(\beta)| \leq \beta^{-n\eta}\} = \bigcap_{N=1}^\infty \bigcup_{n=N}^\infty \bigcup_{\omega \in \Omega_n(\beta_1,\beta_2)} I^P_n(\omega)$$
where $$\Omega_n(\beta_1,\beta_2)=\{\omega=(\omega_1,\ldots,\omega_n) \in \Lambda_n:\ \exists\ \beta \in (\beta_1,\beta_2],\ s.t.\ \varepsilon_1(1,\beta)=\omega_1,\ldots,\varepsilon_n(1,\beta)=\omega_n\ {\rm{and}}\ |I^P_n(\omega)|\leq \beta^{-n\eta}\}.$$

At the moment, we estimate the number of $\Omega_n(\beta_1,\beta_2)$. For every $\omega \in \Omega_n(\beta_1,\beta_2)$, by Lemma \ref{EST}, there exists $\beta \in (\beta_1,\beta_2]$ which verifies that
\begin{equation}\label{eva}
\begin{array}{rcl}
\beta^{-n\eta} &\geq & |I^P_n(\omega)| \\
&\geq & \frac{(\underline{\beta}_n-1)^2}{\underline{\beta}_n} \overline{\beta}_n^{-n}\left(\frac{\omega_{t_n(\beta)+1}}{\overline{\beta}_n}+ \cdots+\frac{\omega_{\tau_n(\beta)}+1}{\overline{\beta}_n^{\tau_n(\beta)-t_n(\beta)}}\right).
\end{array}
\end{equation}

We claim that for every $\omega \in \Omega_n(\beta_1,\beta_2)$, there exists $N_1$, for all $n\geq N_1$, we have $$\omega_{t_n(\beta)+1}=\cdots=\omega_{t_n(\beta)+k_n(\beta)}=0$$where $k_n(\beta)=\lfloor n\eta \log_{\overline{\beta}_n}\beta\rfloor-n+\left \lfloor\log_{\overline{\beta}_n}\frac{(\underline{\beta}_n-1)^2}{\underline{\beta}_n}\right \rfloor$. If the claim is not true, then for every $N\geq 1$, there exists $n>N$, such that
$$\frac{(\underline{\beta}_n-1)^2}{\underline{\beta}_n} \overline{\beta}_n^{-n}\left(\frac{\omega_{t_n(\beta)+1}}{\overline{\beta}_n}+ \cdots+\frac{\omega_{\tau_n(\beta)}+1}{\overline{\beta}_n^{\tau_n(\beta)-t_n(\beta)}}\right)> \frac{(\underline{\beta}_n-1)^2}{\underline{\beta}_n} \overline{\beta}_n^{-n}\frac{1}{\overline{\beta}_n^{k_n}}=\beta^{-n\eta} .
$$ A contradiction with (\ref{eva}).

Since $\lim\limits_{n\rightarrow\infty} \overline{\beta}_n=\lim\limits_{n\rightarrow\infty} \underline{\beta}_n=\beta,$ there exists $N_2$, such that for every $n\geq N_2$, we have $\underline{\beta}_n\geq \beta_1$ and $\overline{\beta}_n\leq \beta_2$. We can choose another large enough integer  $N_2$ such that for every $n\geq N_2$, we have $\lfloor n\eta \log_{\overline{\beta}_n}\beta\rfloor \geq \lfloor n\eta \rfloor \geq n-\left \lfloor \log_{\beta_2}\frac{(\beta_1-1)^2}{\beta_1}\right\rfloor+1$.

Let $u_n=\lfloor n\eta \rfloor-n+\left \lfloor\log_{\beta_2}\frac{(\beta_1-1)^2}{\beta_1}\right \rfloor$ and $M\geq\max\{N_1,N_2\}$.  The above discussion indicates that $1\leq u_n\leq k_n(\beta)$ for every $n\geq M$. Therefore, for every $n\geq N$, we have $\Omega_n(\beta_1,\beta_2)\subset \Lambda_{n,u_n}(\beta_1,\beta_2)$ where $\Lambda_{n,u_n}(\beta_1,\beta_2)$ is defined as (\ref{nk}). Then by Lemma \ref{cardinality}, we have
$$\sharp \Omega_n(\beta_1,\beta_2)\leq \sharp \Lambda_{n,u_n}(\beta_1,\beta_2)\leq (\beta_2-1)\beta_2^{n-u_n+1}=\frac{\beta_1}{(\beta_1-1)^2(\beta_2-1)} \beta_2^{2n-\lfloor n\eta \rfloor+1}.$$
Therefore, we have for any $s>\frac{(2-\eta)\log_{\beta_1}\beta_2}{\eta}\rightarrow \frac{(2-\delta)\log_{\beta_1}\beta_2}{\delta}$, we have $$\begin{array}{rcl}\mathcal{H}^s\left(F^P_{\delta}(\beta_1,\beta_2)\right) &\leq & \liminf\limits_{N \rightarrow \infty}\sum\limits_{n=N}^\infty \beta_1^{-n\eta s}\frac{\beta_1}{(\beta_1-1)^2(\beta_2-1)}\beta_2^{2n-\lfloor n\eta\rfloor+1} \\
&= & \liminf\limits_{N\rightarrow\infty}\sum\limits_{n=N}^\infty \frac{\beta_1}{(\beta_1-1)^2(\beta_2-1)}\beta_1^{-n\eta s+(2n-\lfloor n\eta\rfloor+1)\log_{\beta_1}\beta_2}\\&<& \infty.
\end{array}$$ So $\dim_{\rm{H}}F_{\delta}^P(\beta_1,\beta_2)\leq\frac{(2-\delta)\log_{\beta_1}\beta_2}{\delta}$.

For any $\varepsilon>0$, we partition the parameter space $(1,+\infty)$ into $\{(a_n,a_{n+1}]:-\infty < n< \infty\}$ with $\frac{\log a_{n+1}}{\log a_n}<1+\varepsilon$  and $a_n\rightarrow 1$ as $n\rightarrow -\infty$. Then $$F_{\delta}^P=\bigcup_{n=-\infty}^\infty F_{\delta}^P(a_n,a_{n+1}).$$ By the $\sigma$-stability of the Huasdorff dimension, it suffices to get that $\dim_{\rm{H}}F_{\delta}^P\leq\frac{(2-\delta)}{\delta}$.

\subsection{Lower bound of $\dim_{\rm{H}}F_{\delta}^P$ }

This section devotes to estimating the lower bound of $\dim_{\rm{H}}F_{\delta}^P$. We only need to find out a Cantor subset $E$ being contained in $F_{\delta}^P$ satisfies $\dim_{\rm{H}}E\geq \frac{2-\delta}{\delta}$. A classical technique of estimating the lower bound is applying the following Mass distribution principle.
\begin{lemma}[Falconer \cite{FE}]\label{mass}
Let $E$ be a Borel subset of $\R^d$ and $\mu$ be a Borel measure with support $E$. Suppose for any $x\in E$,$$ \liminf_{r\rightarrow 0} \frac{\log \mu(B(x,r))}{\log r}\geq s.$$ Then we have $\dim_{\rm H} E \geq s$.
\end{lemma}

Typically, we divide four steps to complete our result: we first construct a Cantor subset $E$ of $F_{\delta}^P$ and then we define a measures or mass distribution $\mu$ with $\mu(E)>0$. After that, we estimate the $H\ddot{o}lder$ exponent of the measure $\mu$. Finally, we obtain our result by using the above theorem.

\subsubsection{Cantor subset of  $F_{\delta}^P$}
Fix $0<\zeta<\frac{2-\delta}{2\delta}$, let $N> 1$ be an integer. Given $n\geq 1$, let $$U_n=\{u=(u_1,\ldots,u_n):u_i\in\{1,\ldots,N-1\},\ \forall\  1\leq i\leq n\}$$ which is a set with $(N-1)^n$ elements.

Now we recursively construct the Cantor subset $E$ of $F_{\delta}^P$.\\
\textbf{\textbf{Step \uppercase\expandafter{\romannumeral1}}} The $0$th generation of $E$

Let $$\varepsilon^{(0)}=(N,N)\ \ {\rm and} \ \ E_0=\{\varepsilon^{(0)}\}.$$
We define the $0$th generation of $E$ as $\mathcal{E}_0=\left\{I_2^P(\varepsilon^{(0)}):\varepsilon^{(0)}\in E_0\right\}.$ For the sake of simplicity, the family of $E_0$ is also said to be the $0$th generation of $E$  without any ambiguity in the remaining parts of this paper.\\
\textbf{Step \uppercase\expandafter{\romannumeral2}} The $1$th generation of $E$

Let $m_0=2$ be the length of $\varepsilon^{(0)}\in E_0$. We choose an integer $n_1$ which is large enough such that $n_1\gg m_0$. Denote $a_1=n_1-2$. We collect a family of self-admissible and non-recurrent words beginning with $\varepsilon^{(0)}$:
$$A(\varepsilon^{(0)})=\left\{(\varepsilon^{(0)},u):u \in U_{a_1}\right\}.$$
The self-admissibility and non-recurrence of the elements in $A(\varepsilon^{(0)})$ are guaranteed by the criterion of a self-admissible and non-recurrent word.

After that, let
$$b_1=\left\lfloor\frac{1}{\zeta}\frac{\delta-1}{\delta}n_1\right\rfloor+1\geq 1,$$
and
$$c_1=\left\lfloor\left(\frac{1}{\zeta}\frac{2-\delta}{\delta}-2\right)n_1\right\rfloor+1\geq 1.$$
The second inequality (\ref{2}) is ensured by the choice of $\zeta$, then we give the first generation of $E$:
$$E_1=\left\{\varepsilon^{(1)}=(v,0^{b_1-1},N,u',v,0^{b_1-1}): v\in A(\varepsilon^{(0)}), u'\in U_{c_1}\right\}.$$  At this moment, we give some simple observation on the elements in $ E_1$.
\begin{remark}
(1) For every $\omega\in E_1$, by the criterion of self-admissibility and the choice of elements in $U_n$, we easily get that $\omega$ is self-admissible.\\
(2) For every $\omega\in E_1$, recall the definition of $\tau(\omega)$ (\ref{tau}) and $t(\omega)$ (\ref{t}), we obtain that $\tau(\omega)=n_1+b_1+c_1$ and $t(\omega)=n_1+b_1-1$.\\
(3) For every $\omega\in E_1$, the construction of $\omega$ and the criterion of self-admissibility and non-recurrence ensure that the word $(\omega,u)$ is self-admissible and non-recurrent for every $u\in U_n(n\geq 1).$
\end{remark}

Thus, by the analysis of the words in $E_1$, we define:
$$\mathcal{E}_1=\bigcup_{\varepsilon^{(1)}\in E_1}I_{2n_1+2b_1+c_1-1}^P(\varepsilon^{(1)}).$$ Let $m_1=2n_1+2b_1+c_1-1$.\\
\textbf{Step \uppercase\expandafter{\romannumeral3}} The $k$th generation from the $k-1$th generation of $E$

Suppose that the $k-1$th generation of $E$ has been well defined, we write it as $E_{k-1}$. Moreover, the elements in $E_{k-1}$ can concatenated with any word in $U_n$ to be a self-admissible and non-recurrent word.  We give the construction of the $k$th generation of $E$.

Denote $m_{k-1}$ as the length of $\varepsilon^{(k-1)}\in E_{k-1}$. Choose an integer $n_k\in \N$ such that $n_k\gg m_{k-1}$, write $a_k=n_k-m_{k-1}$, then we collect a family of self-admissible and non-recurrent words beginning with $\varepsilon^{(k-1)}$:
$$A(\varepsilon^{(k-1)})=\left\{(\varepsilon^{(k-1)},u):u \in U_{a_k}\right\}.$$
After that, let
\begin{equation}\label{1}
b_k=\left\lfloor\frac{1}{\zeta}\frac{\delta-1}{\delta}n_k\right\rfloor+1\geq 1,
\end{equation}
and
\begin{equation}\label{2}
c_k=\left\lfloor(\frac{1}{\zeta}\frac{2-\delta}{\delta}-2)n_k\right\rfloor+1\geq 1.
\end{equation}
The second inequality (\ref{2}) is ensured by the choice of $\zeta$, then we define the $k$th generation of $E$:
$$E_k=\left\{\varepsilon^{(k)}=(v,0^{b_k-1},N,u',v,0^{b_k-1}): v\in A(\varepsilon^{(k-1)}), u'\in U_{c_k}\right\}.$$  We now give some simple observation on the elements in $E_k$.
\begin{remark}\label{3}
(1) For any $\omega\in E_k$, by the criterion of self-admissibility and the choice of elements in $U_n$, we easily get that $\omega$ is self-admissible.\\
(2) For any $\omega\in E_k$, recall the definition of $\tau(\omega)$ (\ref{tau}) and $t(\omega)$ (\ref{t}), the construction of $E_k$ implies that $\tau(\omega)=n_k+b_k+c_k$ and $t(\omega)=n_k+b_k-1$.\\
(3) For any $\omega\in E_k$,the construction of $\omega$ and the criterion of self-admissibility and non-recurrence ensure that the word $(\omega,u)$ is self-admissible and non-recurrent for every $u\in U_n(n\geq 1).$
\end{remark}

Thus, by the analysis of the words in $E_k$, let
$$\mathcal{E}_k=\bigcup_{\varepsilon^{(k)}\in E_k}I_{2n_k+2b_k+c_k-1}^P\left(\varepsilon^{(k)}\right).$$ Let $m_k=2n_k+2b_k+c_k-1$.
Repeating the procedure mentioned above, we obtain a nested sequence $\{\mathcal{E}_k\}_{k\geq1}$ which is composed by the cylinders in parameter space. We finally get the Cantor set which is defined as $$E=\bigcap_{k=1}^\infty\bigcup_{\varepsilon^{(k)}\in E_k} I_{|\varepsilon^{(k)}|}^P\left(\varepsilon^{(k)}\right)=\bigcap_{k=1}^\infty\bigcup_{\varepsilon^{(k)}\in E_k} I_{2n_k+2b_k+c_k-1}^P\left(\varepsilon^{(k)}\right).$$

Let $C:=\frac{(N-1)^2}{N}$, we get that for every $\beta \in E$ and $n\geq1$, we have $N<\underline{\beta}_n\leq N+1$, so $C_n(\beta)=\frac{(\underline{\beta}_n-1)^2}{\underline{\beta}_n}\geq C$. Now we show that $E$ is the subset of $I_{\delta}^P$ as the following result:
\begin{lemma}
$E \subset F_{\delta}^P$.
\end{lemma}
\begin{proof}
For every $\beta \in E$. Let $m_k=2n_k+2b_k+c_k-1$, by the construction of $E$, there exists $\varepsilon^{(k)}\in E_k$ such that $\varepsilon(1,\beta)|_{m_k}=\varepsilon^{(k)}$. Now we estimate the length of $I_n^P(\beta)$ when $m_{k-1} <n \leq m_k$ and give the upper density of $\beta$, we distinguish four cases for discussion.\\
(1) When $m_{k-1} <n \leq n_k+b_k+c_k$ and $n \neq n_k+b_k$, by the construction of $E_k$, recall the definition of $\tau_n(\beta)$ (\ref{taun}), we get that $\tau_n(\beta)=n$, so Lemma \ref{EST} gives that
\begin{equation}\label{l1}
|I_n^P(\beta)|\geq C_n(\beta)\overline{\beta}_n^{-n}\geq C\overline{\beta}_n^{-n}.
\end{equation}
 Consequently, we get that $$\frac{-\log_\beta|I_n^P(\beta)|}{n}\leq -\frac{\log C}{n}+\frac{\log\overline{\beta}_n}{\log\beta}\rightarrow 1$$ as $k\rightarrow \infty$.\\
(2) When $n=n_k+b_k$, we have $\tau_n(\beta)=n-1$, $t_n(\beta)=1$, so $\varepsilon_{t_n(\beta)+1}=N\geq 1$, by Lemma \ref{EST}, we have
\begin{equation}\label{l2}
|I_n^P(\beta)|\geq C_n(\beta)\overline{\beta}_n^{-n}\left(\frac{\varepsilon_{t_n(\beta) +1}(1,\beta)}{\overline{\beta}_n}+ \cdots + \frac{\varepsilon_{\tau_n(\beta)}(1,\beta)+1}{\overline{\beta}_n^{(\tau_n(\beta)-t_n(\beta))}}\right) \geq C\overline{\beta}_n^{-n-1}.
\end{equation}
This indicates that $$\frac{-\log_\beta|I_n^P(\beta)|}{n}\leq -\frac{\log C}{n}+\frac{(n+1)\log\overline{\beta}_n}{n\log\beta}\rightarrow 1$$ as $k\rightarrow \infty$.\\
(3) When $n_k+b_k+c_k \leq n < 2n_k+b_k+c_k$, by the construction of $E_k$, recall the definition of $\tau_n(\beta)$ (\ref{taun}) and $t_n(\beta)$ (\ref{tn}), we get that $\tau_n(\beta)=n_k+b_k+c_k$ and $0\leq t_n(\beta)=n-(n_k+b_k+c_k)\leq n_k-1$, so we have  $\varepsilon_{t_n(\beta)+1}\geq 1$ when $m_{k-1}\leq n \leq n_{k-1}$ and  $\varepsilon_{t_n(\beta)+m_{k-1}}\geq 1$ when $0\leq n \leq m_{k-1}$, we immediate get by Lemma \ref{EST} that
\begin{equation}\label{l3}
|I_n^P(\beta)|\geq C_n(\beta)\overline{\beta}_n^{-n}\left(\frac{\varepsilon_{t_n(\beta) +1}(1,\beta)}{\overline{\beta}_n}+ \cdots + \frac{\varepsilon_{\tau_n(\beta)}(1,\beta)+1}{\overline{\beta}_n^{(\tau_n(\beta)-t_n(\beta))}}\right) \geq C\overline{\beta}_n^{-n-m_{k-1}}.
\end{equation}
Similar to the case (2), we get that $$\frac{-\log_\beta|I_n^P(\beta)|}{n}\leq -\frac{\log C}{n}+\frac{(n+m_{k-1})\log\overline{\beta}_n}{n\log\beta}\rightarrow  1$$ as $k\rightarrow \infty$.\\
(4) When $2n_k+b_k+c_k \leq n \leq 2n_k+2b_k+c_k-1=m_k$, by the construction of $E_k$, we get that $\tau_n(\beta)=n_k+b_k+c_k$ and $t_n(\beta)=n-(n_k+b_k+c_k)$, thus, Lemma \ref{EST} indicates that
\begin{equation}\label{l4}
\begin{array}{rcl}
|I_n^P(\beta)|&\geq &C_n(\beta)\overline{\beta}_n^{-n}\left(\frac{\varepsilon_{t_n(\beta) +1}(1,\beta)}{\overline{\beta}_n}+ \cdots + \frac{\varepsilon_{\tau_n(\beta)}(1,\beta)+1}{\overline{\beta}_n^{(\tau_n(\beta)-t_n(\beta))}}\right)\\& \geq & C_n(\beta)\overline{\beta}_n^{-n}\frac{1}{\overline{\beta}_n^{n_k+b_k+c_k-(n-n_k-b_k-c_k)}}\\&\geq& C\overline{\beta}_n^{-(2n_k+2b_k+c_k)}.
\end{array}
\end{equation}
Consequently, the above inequality combines with (\ref{1}), (\ref{2}) gives that $$\frac{-\log_\beta|I_n^P(\beta)|}{n} \leq -\frac{\log C}{n}+\frac{(2n_k+2b_k+c_k)\log\overline{\beta}_n}{(2n_k+b_k+c_k)\log\beta}\rightarrow \delta $$ as $k\rightarrow \infty$.

Now we find a subsequence of $n\in\N$ satisfies that its upper density reach to $\delta$. In fact, let $\ell_k=2n_k+b_k+c_k$, when $n=\ell_k$, we have $\tau_n(\beta)=n_k+b_k+c_k$. Then it follows from Lemma \ref{S} and Lemma \ref{max} that:
$$1=\frac{\varepsilon_1(1,\beta)}{\underline{\beta}_n}+\cdots+\frac{\varepsilon_n(1,\beta)}{\underline{\beta}_n^n},$$
and
$$1=\frac{\varepsilon_1(1,\beta)}{\overline{\beta}_n}+\cdots+\frac{\varepsilon_{\tau_n(\beta)}(1,\beta)}{\overline{\beta}_n^{\tau_n(\beta)}}=
\frac{\varepsilon_1(1,\beta)}{\overline{\beta}_n}+\cdots+\frac{\varepsilon_{n}(1,\beta)}{\overline{\beta}_n^{n}}+
\frac{\varepsilon_{2n_k+2b_k+c_k}(1,\beta)}{\overline{\beta}_n^{2n_k+2b_k+c_k}}+\cdots+ \frac{\varepsilon_{2\tau_n(\beta)}(1,\beta)+1}{\overline{\beta}_n^{2\tau_n(\beta)}}.$$
As a consequence, we have
$$\begin{array}{rcl}
|I_n^P(\beta)|&= &\overline{\beta}_n-\underline{\beta}_n \\& = &
\left(\varepsilon_1(1,\beta)+\cdots +\frac{\varepsilon_{n}(1,\beta)}{\overline{\beta}_n^{n-1}}+ \frac{\varepsilon_{2n_k+2b_k+c_k}(1,\beta)}{\overline{\beta}_n^{2n_k+2b_k+c_k-1}} +\cdots+ \frac{\varepsilon_{2\tau_n(\beta)}(1,\beta)+1}{\overline{\beta}_n^{2\tau_n(\beta)-1}}\right)- \left(\varepsilon_1(1,\beta)+\cdots+\frac{\varepsilon_n(1,\beta)}{\underline{\beta}_n^{n-1}}\right)
\\& \leq & \frac{\varepsilon_{2n_k+2b_k+c_k}(1,\beta)}{\overline{\beta}_n^{2n_k+2b_k+c_k-1}}+\cdots+ \frac{\varepsilon_{2\tau_n(\beta)}(1,\beta)+1}{\overline{\beta}_n^{2\tau_n(\beta)-1}}\\&\leq& \frac{N}{\overline{\beta}_n^{2n_k+2b_k+c_k-1}}.
\end{array}$$
That is, $$\lim_{k\rightarrow\infty}\frac{-\log_\beta|I_{\ell_k}^P(\beta)|}{\ell_k} \geq \lim_{k\rightarrow\infty}\frac{-\log N+(2n_k+2b_k+c_k-1)\log\overline{\beta}_n}{(2n_k+b_k+c_k)\log\beta}= \delta.$$

From the discussion above, we get that $\beta\in F_{\delta}^P.$ The desired conclusion follows.
\end{proof}
\subsubsection{Mass distribution supported on $E$}
In this section, we mainly define a measure supported on $E$. Such a measure is given according to the cylinders which have non-empty intersection with $E$. For all $\beta\in E$, let $\{I_n^P(\beta)\}_{n\geq1}$ be the cylinders which contain $\beta$ and denote $$\varepsilon(1,\beta)=(\varepsilon^{(k-1)},u_1,\ldots,u_{a_k},0^{b_k-1},N,u_1',\ldots,u_{c_k},\varepsilon^{(k-1)},u_1,\ldots,u_{a_k},0^{b_k-1},\ldots),$$

Note that the order of $\varepsilon^{(k-1)}$ is $m_k:=2n_k+2b_k+c_k-1$ where $m_0=2$. Define $$\mu(I_2^P(\varepsilon^{(0)}))=\mu(I_2^P(N,N))=1.$$ For all $n\geq 2$ and $k\geq 1$, let
$$\mu(I^P_n(\beta))= \left\{
\begin{aligned}
\frac{1}{(N-1)^{n-m_{k-1}}}\mu\left(I_{m_{k-1}}^P(\varepsilon^{(k-1)})\right) & , &{\rm when}\ \ m_{k-1}< n\leq n_k ; \\
\mu\left(I_{m_{k-1}}^P(\varepsilon^{(k-1)},u_1,\ldots,u_{a_k})\right)& , & {\rm when}\ \ n_k < n\leq n_k+b_k;\\
\frac{1}{(N-1)^{n-n_k-b_k}}\mu\left(I_{n_k}^P(\varepsilon^{(k-1)},u_1,\ldots,u_{a_k},0^{b_k-1},N)\right)& , & {\rm when}\ \ n_k+b_k< n\leq n_k+b_k+c_k; \\
\mu\left(I_{n_k+b_k+c_k}^P(\varepsilon^{(k-1)},u_1,\ldots,u_{a_k},0^{b_k-1},N,u_1',\ldots,u'_{c_k})\right)& , & {\rm when}\ \ n_k+b_k+c_k < n\leq m_k;
\end{aligned}
\right.$$

More precisely,
$$\mu(I^P_n(\beta))= \left\{
\begin{aligned}
\frac{1}{(N-1)^{n-m_{k-1}+\sum\limits_{i=1}^{k-1}(a_i+c_i)}} & , &{\rm when}\ \ m_{k-1}< n \leq n_k; \\
\frac{1}{(N-1)^{a_k+\sum\limits_{i=1}^{k-1}(a_i+c_i)}} & , & {\rm when}\ \ n_k < n \leq n_k+b_k;\\
\frac{1}{(N-1)^{a_k+n-n_k-b_k+\sum\limits_{i=1}^{k-1}(a_i+c_i)}} & , & {\rm when}\ \ n_k+b_k< n\leq n_k+b_k+c_k; \\
\frac{1}{(N-1)^{\sum\limits_{i=1}^k(a_i+c_i)}}& , & {\rm when}\ \ n_k+b_k+c_k < n\leq m_k ;
\end{aligned}
\right.$$

By the definition of $\mu$ supported on $E$. We directly get the following result which will serve as a key point to estimate the lower bound of the lower limit of $\frac{\log \mu(B(\beta,r))}{\log r}$ where $B(\beta,r)$ is a arbitrarily small ball with non-empty intersection with $E$.
\begin{lemma}\label{e}
Let $\beta \in E$, then we have
$$\liminf_{n\rightarrow\infty}\frac{\log\mu\left(I_n^P(\beta)\right)}{\log|I_{n+1}^P(\beta)|}\geq \frac{\log(N-1)}{\log(N+1)}\left(\frac{2-\delta}{\delta}-\zeta\right).$$
\end{lemma}
\begin{proof}
In order to get the result, we only need to estimate the value of $\frac{\log\mu\left(I_n^P(\beta)\right)}{\log|I_{n+1}^P(\beta)|}$ when $m_{k-1}\leq n< m_k$, where $m_k=2n_k+2b_k+c_k-1$. This is discussed by four cases.\\
(1) When $m_{k-1}\leq n< n_k$, by (\ref{l1}), then
$$\begin{array}{rcl}
\frac{\log\mu\left(I_n^P(\beta)\right)}{\log|I_{n+1}^P(\beta)|}&\geq& \frac{\left(n-m_{k-1}+\sum\limits_{i=1}^{k-1}(a_i+c_i)\right)\log(N-1)}{-\log C+n\log \overline{\beta}_n}\\&\geq& \frac{\left(n-m_{k-1}+\sum\limits_{i=1}^{k-1}(a_i+c_i)\right)\log(N-1)}{-\log C+n\log (N+1)}\\&\geq& \frac{\left(m_{k-1}-m_{k-1}+\sum\limits_{i=1}^{k-1}(a_i+c_i)\right)\log(N-1)}{-\log C+m_{k-1}\log (N+1)}\\&\rightarrow& \frac{\log(N-1)}{\log(N+1)}\left(\frac{2-\delta}{\delta}-\zeta\right),
\end{array}$$as $k\rightarrow\infty$, and the last inequality is ensured by the increasing of the function $f(n)=\frac{\left(n-m_{k-1}+\sum\limits_{i=1}^{k-1}(a_i+c_i)\right)\log(N-1)}{-\log C+n\log (N+1)}$. \\
(2) When $n_k \leq n < n_k+b_k$, it follows from (\ref{l1}) and (\ref{l2}) that
$$\begin{array}{rcl}
\frac{\log\mu\left(I_n^P(\beta)\right)}{\log|I_{n+1}^P(\beta)|}&\geq& \frac{\left(a_k+\sum\limits_{i=1}^{k-1}(a_i+c_i)\right)\log(N-1)}{-\log C+(n+1)\log \overline{\beta}_n}\\&\geq& \frac{\left(a_k+\sum\limits_{i=1}^{k-1}(a_i+c_i)\right)\log(N-1)}{-\log C+(n_k+b_k+1)\log (N+1)}\\ &\rightarrow& \frac{\log(N-1)}{\log(N+1)}\frac{\zeta\delta}{\zeta\delta+\delta-1}\geq \frac{\log(N-1)}{\log(N+1)} \frac{2-\delta}{\delta},
\end{array}$$as $k\rightarrow\infty$, where the last inequality is guaranteed by the choice of $\zeta$. \\
(3) When $n_k+b_k\leq n<n_k+b_k+c_k$, similarly, by (\ref{l1}), we get that
$$\begin{array}{rcl}
\frac{\log\mu\left(I_n^P(\beta)\right)}{\log|I_{n+1}^P(\beta)|}&\geq& \frac{\left(a_k+n-n_k-b_k+\sum\limits_{i=1}^{k-1}(a_i+c_i)\right)\log(N-1)}{-\log C+n\log \overline{\beta}_n}\\&\geq& \frac{\left(a_k+\sum\limits_{i=1}^{k-1}(a_i+c_i)\right)\log(N-1)}{-\log C+(n_k+b_k)\log (N+1)}\\&\rightarrow& \frac{\log(N-1)}{\log(N+1)}\frac{\zeta\delta}{\zeta\delta+\delta-1}\geq \frac{\log(N-1)}{\log(N+1)} \frac{2-\delta}{\delta},
\end{array}$$as $k\rightarrow\infty.$\\
(4) When $n_k+b_k+c_k\leq n< m_k$, then (\ref{l3}) and (\ref{l4}) give that
$$\begin{array}{rcl}
\frac{\log\mu\left(I_n^P(\beta)\right)}{\log|I_{n+1}^P(\beta)|}&\geq& \frac{\left(\sum\limits_{i=1}^{k}(a_i+c_i)\right)\log(N-1)}{-\log C+(2n_k+2b_k+c_k)\log \overline{\beta}_n}\\&\geq& \frac{\left(\sum\limits_{i=1}^{k}(a_i+c_i)\right)\log(N-1)}{-\log C+(2n_k+2b_k+c_k)\log (N+1)}\\&\rightarrow& \frac{\log(N-1)}{\log(N+1)}\left(\frac{2-\delta}{\delta}-\zeta\right),
\end{array}$$as $k\rightarrow\infty.$\\

From what has been discussed, we have $\liminf\limits_{n\rightarrow\infty}\frac{\log\mu\left(I_n^P(\beta)\right)}{\log|I_{n+1}^P(\beta)|}\geq \frac{\log(N-1)}{\log(N+1)}\left(\frac{2-\delta}{\delta}-\zeta\right).$
\end{proof}
\subsubsection{Measure of Balls}
Now we give an estimation on the measure of arbitrary balls $B(\beta,r)$ with non-empty intersection with $E$ and $r$ is small enough. Before doing that, we first refine the cylinders containing some $\beta \in E$. For any $\beta \in E$ and $n\in\N$, let
$$J_n(\beta)= \left\{
\begin{aligned}
I_n^P(\beta) & , &{\rm when}\ \ m_{k-1}<n \leq n_k\ {\rm \ for\ some\ k\geq 1}; \\
I_{n_k}^P(\beta)& , & {\rm when}\ \ n_k < n\leq n+b_k\ {\rm \ for\ some\ k\geq 1};\\
I_n^P(\beta)& , & {\rm when}\ \ n_k+b_k<n\leq n_k+b_k+c_k\ {\rm \ for\ some\ k\geq 1}; \\
I_{n_k+b_k+c_k}^P(\beta)& , & {\rm when}\ \ n_k+b_k+c_k< n\leq m_k\ {\rm \ for\ some\ k\geq 1}.
\end{aligned}
\right.$$

Let $B(\beta,r)$ be a ball with non-empty intersection with $E$ and $r$ be small enough. Assume that $n$ is an integer satisfies $$|J_{n+1}(\beta)|\leq r<|J_n(\beta)|.$$ Then there exists a $k$ such that $n=m_{k-1}\leq n< m_k$. Since the lengths of $J_n(\beta)$ and $J_{n+1}(\beta)$ are different, we get that
\begin{equation}\label{n}
m_{k-1}=2n_{k-1}+2b_{k-1}+c_{k-1} < n \leq n_k\  {\rm or}\ n_k+b_k<n\leq n_k+b_k+c_k.
\end{equation}

An important point we should notice is that$$\mu(J_n(\beta))=\mu(I_n^P(\beta)),$$for every $n\geq 1$. This means that all basic intervals $J_n$ have the same order are of equal $\mu$-measure. Hence, in order to bound the measure of the ball $B(\beta,r)$ given above, now we estimate the number $\mathcal{N}$ of the basic intervals with non-empty intersection with the ball  $B(\beta,r)$. Note that by (\ref{l1}), for all $n$ satisfying (\ref{n}), we have $|J_n(\beta)|\geq C\overline{\beta}_n^{-n}\geq C (N+1)^{-n}$. Since $r\leq |J_n(\beta)|\leq \overline{\beta}_n^{-n+1}\leq N^{-n+1}$, we obtain that $$\mathcal{N}\leq \frac{2r}{ C (N+1)^{-n}}+2\leq \frac{2N^{-n+1}}{C (N+1)^{-n}}+2\leq C_1N^{-n}(N+1)^n.$$
As a consequence, we have
\begin{equation}\label{m1}
\mu(B(\beta,r))\leq C_1N^{-n}(N+1)^n\mu(I_n^P(\beta)).
\end{equation}

To estimate the lower limit of $\frac{\log \mu(B(x,r))}{\log r}$, the next step is giving a lower bound for $r$. As a matter of fact, when n satisfies $m_{k-1}=2n_{k-1}+2b_{k-1}+c_{k-1} < n \leq n_k\  {\rm or}\ n_k+b_k<n\leq n_k+b_k+c_k-1$, by the construction of $E$, we have
\begin{equation}\label{m2}
r\geq|J_{n+1}(\beta)|\geq C\overline{\beta}_n^{-n-1}\geq C(N+1)^{-n-1}.
\end{equation} When $n=n_k+b_k+c_k-1$, we get
\begin{equation}\label{m3}
r\geq|J_{n+1}(\beta)|\geq C\overline{\beta}_n^{-2n_k-2b_k-c_k}\geq C(N+1)^{-2n_k-2b_k-c_k}.
\end{equation}
Thus, by Lemma \ref{e} and the above inequalities(\ref{m1}) (\ref{m2}) (\ref{m3}), we have
$$\liminf_{r\rightarrow 0}\frac{\log \mu(B(\beta,r))}{\log r}\geq\left(\frac{\log N-\log(N+1)}{\log(N+1)}+ \frac{\log(N-1)}{\log(N+1)}\right)\left(\frac{2-\delta}{\delta}-\zeta\right).$$

Finally, the mass distribution principle(Lemma \ref{mass}) gives that $$\dim_{\rm{H}}E\geq\left(\frac{\log N-\log(N+1)}{\log(N+1)}+ \frac{\log(N-1)}{\log(N+1)}\right)\left(\frac{2-\delta}{\delta}-\zeta\right).$$ Let $\zeta\rightarrow 0$ and then $N\rightarrow \infty$, we obtain that $$\dim_{\rm{H}}E\geq \frac{2-\delta}{\delta}.$$


\begin{thebibliography}{99}

\bibitem{B} F. Blanchard, {\it $\beta$-expansions and symbolic dynamics}, Theoret. Comput. Sci 65 (1989), no. 2, 131-141.

\bibitem{BW} Y. Bugeaud and B.-W. Wang, {\it Distribution of full cylinders and the Diophantine properties of the orbits in $\beta$-expansions}, J. Fractal Geom. 1 (2014), no. 2, 221-241.


\bibitem{FE} K.J. Falconer, {\it Fractal Geometry: Mathematical Foundations and Applications}, John Wiley and Sons, Ltd., Chichester, 1990.

\bibitem{AB} A.-H. Fan and B.-W. Wang, {\it On the lengths of basic intervals in beta expansions}, Nonlinearity 25 (2012), no. 5, 1329-1343.

\bibitem{LP} B. Li, T. Persson, B.-W. Wang and J. Wu, {\it Diophantine approximation of the orbit of $1$ in the dynamical system of beta expansion}, Math. Z. 276 (2014), no.3-4, 799-827.

\bibitem{BWJ} B. Li and J. Wu, {\it Beta-expansion and continued fraction expansion}, J. Math. Anal. Appl. 339 (2008), no. 2, 1322-1331.

\bibitem{LW} F. L\"{u} and J. Wu, {\it Diophantine analysis in beta-dynamical systems and Huasdorff dimension}, Adv. Math. 290 (2016), 919-937.

\bibitem{LWX} M.-Y.  L\"{u}, B.-W. Wang and J. Xu {\it On sums of degrees of the partial quotients in continued fraction expansions of Laurent series}, J. Math. Anal. Appl. 380 (2011), no. 2, 807-813.

\bibitem{P} W. Parry, {\it On the $\beta$-expansions of real numbers},  Acta Math. Acad. Sci. Hungar,  11 (1960), 401-416.

\bibitem{PS} T. Persson and J. Schmeling, {\it Dyadic Diophantine approximation and Katok's horseshoe approximation}, Acta Arith. 132 (2008), no. 3, 205-230.

\bibitem{R} A. R\'{e}nyi, {\it Representations for real numbers and their ergodic properties}, Acta Math. Acad. Sci. Hung. 8 (1957), 477-93.

\bibitem{S} J. Schmeling, {\it Symbolic dynamics for the $\beta$-shifts and self-normal numbers}, Ergod. Thory Dyn. Syst. 17 (1997), 675-694.

\bibitem{TD} D. Thompson, {\it Irregular sets, the $\beta$-transformation and the almost
specification property},  Trans. Amer. Math. Soc. 364 (2012), no. 10, 5395-5414.


\end{thebibliography}
\end{document}